\documentclass{amsart}
\usepackage{amsfonts}
\usepackage{amssymb}
\usepackage{times}

\newtheorem{lemma}{Lemma}

\newtheorem{example}{Example}

\newtheorem{theorem}{Theorem}
  \newtheorem{corollary}{Corollary}
  \newtheorem{proposition}{Proposition}

  \newtheorem{definition}{Definition}

\newfont{\hueca}{msbm10}

\begin{document}
\title{A linear algebra characterization of the semisimplity and the simplicity of arbitrary algebras}

\author{Antonio J.~Calder\'on Mart\'{i}n}
\address{Antonio J. Calder\'on Mart\'{i}n. \newline \indent Dpto. de Matem{\'a}ticas, Universidad de C\'adiz, Puerto Real (Spain).}
\email{{\tt ajesus.calderon@uca.es}}

\thanks{This work was  supported   by the PAI with project number FQM298 and by the Research Project `Operator Theory. an interdisciplinary approach. ProyExcel 00780'. }

\begin{abstract}

We show that an arbitrary   algebra ${ A}$, (of arbitrary dimension, over an
arbitrary base field and   any identity is not suppose for the product), is semisimple if and only if it has zero annihilator and admits a
semi-division linear basis. We also show that $A$ is simple if and only if it has zero annihilator and admits an
$i$-division linear basis.

\medskip


 \medskip
\end{abstract}

\maketitle


\section{Introduction and previous definitions}

We begin by noting that throughout this paper all of the algebras are
considered of arbitrary dimension, over an arbitrary base field ${\mathbb{%
K}}$ and any identity is not suppose for the product. The product in these algebras will be always denoted by
juxtaposition. That is:

\begin{definition}
Let $A$ be a vector space over an arbitrary base field ${%
\mathbb{K}}$. It is said that $A$ is an {\bf algebra} if it  endowed with a bilinear map

$$A \times A \to A$$
$$\hspace{0.3cm}(x,y)\mapsto x y.$$

called the {\bf product} of $A$.

\end{definition}

We recall that an algebra ${A}$ is said to be {\bf simple} if
its product is non-zero and its only ideals are $\{0\}$ and ${A}$. The algebra $A$ is called {\bf semisimple} if it can be
expressed as the direct sum of simple ideals. Also the {\bf Annihilator} of $A$ is the ideal   ${\rm Ann}(A)=\{x\in A: xA=Ax=0\}.$

\medskip

The aim of this paper is to extend the results in \cite{semi} from Lie-type algebras (see Definition \cite[Definition 1]{semi}) to arbitrary algebras. In \cite{semi} it is showed that a Lie-type algebra $L$ is semisimple if and only if ${\rm Ann}(L)=0$ and it admits a weak-division basis (see the below definition) , and that $L$ is simple if and only if ${\rm Ann}(L)=0$ and it  admits an $i$-division basis (see the below definition).

\medskip
\begin{definition}
A linear basis  ${\mathcal B}=\{e_i\}_{i\in I}$   of an arbitrary algebra  ${A}
$ is said to be an {\bf i-division basis} if for any $e_i \in {\mathcal{B}%
}$ and $x \in {A}$ such that

$$\hbox{$0\neq e_ix=c$ or $0\neq xe_i=c$}$$

we have that $e_i,x \in {\mathcal{I}}(c)$, where ${\mathcal{I}}(c)$ denotes
the ideal of ${A}$ generated by $c$. 
\end{definition}

Consider a linear basis ${\mathcal B}=\{e_i\}_{i\in I}$ of an arbitrary algebra $A$.  For any $e_i \in {\mathcal{B}}$,
let us denote by 

$$ {\mathcal{S}}_{e_i}:=\{e_j \in {\mathcal{B}}: e_ie_j \neq
0\} \cup \{e_k \in {\mathcal{B}}: e_ke_i \neq 0\}$$

 and by ${\mathcal{M}}%
_{e_i}$ the linear subspace of ${A}$ generated by ${\mathcal{S}}%
_{e_i}$, being ${\mathcal{M}}_{e_i}=\{0\}$ when ${\mathcal{S}}_{e_i}=
\emptyset$. 

\begin{definition}\label{Harbin50}
\label{weak}\textrm{Let ${\mathcal{B}}=\{e_i: i \in I\}$ be a linear basis
of an arbitrary algebra ${A}$. It is said that ${\mathcal{B}}$ is a
{\bf weak-division basis} of ${A}$ if for any $e_i \in {\mathcal{%
B}}$ and $x \in {\mathcal{M}}_{e_i}$ such that $0\neq e_ix=c$ or $0\neq
xe_i=c$, we have that $e_i,x \in {\mathcal{I}}(c)$. }
\end{definition}

It is clear that any $i$-divison algebra is a weak-division algebra, but the converse does not hols (see Remark 1 in  \cite{semi}).

We note that the existence of a weak-division basis for a non Lie-type algebra $A$ does not characterize, in general, the semisimplicity of $A$ as the next example shows.

\begin{example}\label{ex1}
Consider the complex (non Lie-type algebra) algebra $A= {\mathbb C} \oplus {\mathbb C}$ with basis ${\mathcal B}=\{(1,0), (0,1)\}$ and nonzero products among these elements:

\medskip

$(1,0)(1,0)=(1,0)$

$(0,1)(0,1)=(1,1)$

\medskip 

It is easy to verify that   ${\mathcal B}$ is a weak-division basis and that 
 $A$ is not semisimple.

\end{example}

Hence we need a  stronger condition to characterize the semisimplicity of an arbitrary algebra:

\medskip

Let ${\mathcal{B}}=\{e_i: i \in I\}$ be a linear basis
of an arbitrary algebra ${A}$.
For any $e_i \in {\mathcal{B}}$ we consider  the projection map on $\mathbb{K}e_i$
\begin{equation}\label{Harbin15}
    p_{e_i}:A \to {\mathbb K}
\end{equation}
  and define the set 
$$\hbox{${\mathcal{P}}_{e_i}:=\{e_j \in {\mathcal{B}}: p_{e_i}(e_je_r)\neq
0$ {{for some}} $e_r \in  {\mathcal{B}}
\} \cup \{e_k \in {\mathcal{B}}: p_{e_i}(e_se_k) \neq 0$ {{for some}} $e_s \in  {\mathcal{B}}\}$.}$$
Also we denote by ${\mathcal{P}}_{e_i}{\mathcal{P}}_{e_i}:=\{e_je_k:e_j,e_k \in {\mathcal{P}}_{e_i}\}\subset A.$

\begin{definition}
 The basis  ${\mathcal{B}}=\{e_i: i \in I\}$  of $A$ is said to be  a
{\bf semi-division basis}  if it is a weak-division basis; and for any $e_i \in {\mathcal{%
B}}$ we have that if  $b \in {\mathcal{P}}_{e_i}{\mathcal{P}}_{e_i}$ and $e_j \in S_{e_i}$ are such  that $0\neq be_j=c$ or $0\neq
e_jb=c$, then  $e_j,b \in {\mathcal{I}}(c)$.
\end{definition}


We want to prove the next results:

\medskip 
\hspace{-0.4cm}{\bf Theorem.} {\it Let $A$ be an arbitrary algebra. Then $A$ is semisimple if and only if  ${\rm Ann}({ A}) \neq 0$ and $A$ admits a   semi-division linear basis. }

\medskip 

\hspace{-0.4cm}{\bf Corollary.} {\it Let ${A}$ be an arbitrary  algebra. Then ${A}$ is
simple if and only if ${\rm Ann}({ A})=0$  and ${A}$ admits an i-division
basis.}

\medskip 
\medskip

The paper is organized as follows. In the next section we summarize some of the constructions developed in \cite{semi} that hold not only for Lie-type algebras but also for arbitrary algebras. Finally in Section 3 we prove the above characterization results.

\section{Background}

The proof of the characterization of the (semi)simplicity through the existence of certain linear bases for Lie-type algebras in \cite{semi} follows two steps. 

\bigskip

In the first one, (see \cite[Section 2]{semi}), fixed $A$  an  arbitrary algebra  and ${\mathcal B}=\{e_i\}_{i\in I}$ an arbitrary liner basis  of $A$, it is applied connection techniques on the set  ${\mathcal B}$. We note that these techniques were introduced in \cite{Yo2} for the study of the structure of split Lie algebras and have been developed and applied in many contexts  (see for instance \cite{Yo1, Yo3, Yo4, Yo5, Yo6, Yo7}). The tool consists in defining an adequate equivalence relation $\sim$ on ${\mathcal B}=\{e_i\}_{i\in I}$  in such a way that we get a quotient set 

\begin{equation*}
{\mathcal{B}} / \sim=\{[e_i]: e_i \in {\mathcal{B}}\}.
\end{equation*}

\smallskip

For any $[e_i] \in {\mathcal{B}} / \sim$ it is considered the linear subspace
\begin{equation*}
{A}_{[e_i]}:= \bigoplus\limits_{e_j \in [e_i]} {\mathbb{K}} e_j,
\end{equation*}
and we get that ${A}$ is the  direct sum of
linear subspaces,
\begin{equation}  \label{hel0}
{A}=\bigoplus\limits_{[e_i] \in {\mathcal{B}} / \sim} {A}%
_{[e_i]}.
\end{equation}

The above decomposition enjoys the next property:

\begin{lemma}[Lemma 6 in \cite{semi}]
\label{penu} Let ${A}$ be an arbitrary algebra admitting a
weak-division basis ${\mathcal{B}}=\{e_i:i \in I\}$ and let  ${I}$ be
 an ideal of ${A}$. If some   $e_i \in {%
\mathcal{B}}$ satisfies  $e_i \in {I}$, then ${A}_{[e_i]}
\subset {I}.$
\end{lemma}

\medskip

In the second step, (see again \cite[Section 2]{semi}), it is  considered  the family of linear subspaces of ${A}$,
\begin{equation*}
{\mathcal{F}}:=\{{A}_{[e_{i}]}:[e_{i}]\in {\mathcal{B}}/\sim \},
\end{equation*}%
and it is  introduced a new equivalence relation, denoted by $\approx $, on ${%
\mathcal{F}}$ as follows.

\medskip
\begin{definition}
We will say that ${A}_{[e_i]} \approx {A}_{[e_j]}$ if and
only if either ${A}_{[e_i]} = {A}_{[e_j]}$ or there exists
a  sequence 
\begin{equation*}
([e_1],[e_2],...,[e_n]) 
\end{equation*}
with any $[e_k]\in {\mathcal{B}} / \sim,$ $k \in \{1,...,n\}$ and $n \geq 2$, such that

\medskip

\begin{enumerate}
\item[(i)] $A_{[e_1]}=A_{[e_i]}$ and $A_{[e_n]}=A_{[e_j]}.$

\item[(ii)] ${A}_{[e_k]}{A}_{[e_{k+1}]} + {A}%
_{[e_{k+1}]}{A}_{[e_k]} \neq 0$ for any $k \in \{1,...,n-1\}.$
\end{enumerate}
\end{definition} 

\medskip
Then  it is  considered  the quotient set
\begin{equation*}
{\mathcal{F}} / \approx=\{ [{A}_{[e_{i}]}]: {A}_{[e_{i}]}
\in {\mathcal{F}}\}
\end{equation*}
and for any $[{A}_{[e_{i}]}] \in {\mathcal{F}} / \approx$ it is  denoted by $%
{A}_{[[e_i]]}$ the linear subspace

$${A}_{[[e_i]]}:= \bigoplus\limits_{{A}_{[e_{j}]} \in [%
{A}_{[e_{i}]}]} {A}_{[e_{j}]}. $$

By Equation (\ref{hel0}) and the definition of $\approx$, we clearly have
$${A}=\bigoplus\limits_{[{A}_{[e_{i}]}] \in {\mathcal{F}} /
\approx}{A}_{[[e_{i}]]} $$

and

$${A}_{[[e_{i}]]} {A}_{[[e_{j}]]}=0 $$

when $[{A}_{[e_{i}]}] \neq [{A}_{[e_{j}]}]$.

\bigskip

Finally, it is showed in \cite{semi} that Lie-type identity and the weak-division character of $\mathcal{B}$ allows us to get the desired characterization of the siemisemplicity of al Lie-type algebra. However this characterizetion does not hold for arbitrary algebras as Example \ref{ex1} shows. Hence we need a deeper study to get a characterization of the semiplicity for arbritary algebras. This will be carry out in the next section.

\section{Main results}

Let   $A$ be an arbitrary algebra and ${\mathcal{B}}=\{e_i: i \in I\}$  a linear basis of $A$. By the previous section we can consider  the family of linear subspaces

$${\mathcal{G}}:=\{A_{[[e_{i}]]}:  [{A}_{[e_{i}]}] \in {\mathcal{F}} /
\approx\}.$$
in such a way that

\begin{equation}\label{harbin1}
{A}=\bigoplus\limits_{[{A}_{[e_{i}]}] \in {\mathcal{F}} /
\approx}A_{[[e_{i}]]}
\end{equation}

and 
\begin{equation}\label{orto}
A_{[[e_{i}]]} A_{[[e_{j}]]}=0
\end{equation}
when $A_{[[e_{i}]]}  \neq A_{[[e_{j}]]}$.

\bigskip

 Taking into account Equation (\ref{harbin1}), we can introduce  for any $A_{[[e_{i}]]}$,  the projection map $$\Pi_{A_{[[e_{i}]]}}: A \to A_{[[e_{i}]]}$$
and define the following relation on ${\mathcal{G}}.$

\begin{definition}\label{ane3}\rm

  We  will say that $A_{[[e_{i}]]}\equiv A_{[[e_{j}]]}$ if and
only if either $A_{[[e_{i}]]}= A_{[[e_{j}]]}$ or there exists
a sequence 
$$
(A_{[[e_{1}]]},A_{[[e_{2}]]},...,A_{[[e_{n}]]})
$$
 with $n \geq 2$ and any $A_{[[e_{i}]]} \in {\mathcal G} $ for $i \in \{1,...,n\}$ such that

\begin{itemize}
\item[(i)] $A_{[[e_{1}]]}=A_{[[e_{i}]]}$ and $A_{[[e_{n}]]}=A_{[[e_{j}]]}.$

\item[(ii)] $\Pi_{A_{[[e_{1}]]}}(A_{[[e_{2}]]} A_{[[e_{2}]]}) + \Pi_{A_{[[e_{2}]]}}(A_{[[e_{1}]]}A_{[[e_{1}]]}) \neq 0.$\\
  $\Pi_{A_{[[e_{2}]]}}(A_{[[e_{3}]]}A_{[[e_{3}]]}) + \Pi_{A_{[[e_{3}]]}}(A_{[[e_{2}]]}A_{[[e_{2}]]}) \neq 0.$\\
  $\vdots$\\
  $\Pi_{A_{[[e_{n-1}]]}}(A_{[[e_{n}]]}A_{[[e_{n}]]}) + \Pi_{A_{[[e_{n}]]}}(A_{[[e_{n-1}]]}A_{[[e_{n-1}]]}) \neq 0.$
\end{itemize}

\end{definition}

\begin{lemma}
 The relation  $\equiv$ on ${\mathcal{G}}$
is  an equivalence relation.
\end{lemma}

\begin{proof}
The relation $\equiv$ is reflexive by definition. To verify that $\equiv$ is symmetric, observe that if 
$A_{[[e_{i}]]}\equiv A_{[[e_{j}]]}$ through the  sequence
$$
(A_{[[e_{1}]]},A_{[[e_{2}]]},...,A_{[[e_{n}]]})
$$
with any $A_{[[e_{i}]]} \in {\mathcal G} $ for $i \in \{1,...,n\}$, 
 then   the sequence  $$
(A_{[[e_{n}]]},A_{[[e_{n-1}]]},...,A_{[[e_{1}]]}) 
$$
gives us $A_{[[e_{j}]]}\equiv A_{[[e_{i}]]}.$ 

\medskip
Finally, if $A_{[[e_{i}]]}\equiv A_{[[e_{j}]]}$ by means of 
$$
(A_{[[e_{1}]]},A_{[[e_{2}]]},...,A_{[[e_{n}]]})
$$ and $A_{[[e_{j}]]}\equiv A_{[[e_{k}]]}$ through
$$
(A_{[[e^{\prime}_{1}]]},A_{[[e_{2}^{\prime}]]},...,A_{[[e_{m}^{\prime}]]}),
$$

the sequence 

$$
(A_{[[e_{1}]]},A_{[[e_{2}]]},...,A_{[[e_{n}]]}, A_{[[e_{2}^{\prime}]]},...,A_{[[e_{m}^{\prime}]]})
$$
 gives us $A_{[[e_{i}]]}\equiv A_{[[e_{k}]]},$  taking into account  $A_{[[e_{j}]]}=A_{[[e_{n}]]}=A_{[[e^{\prime}_{1}]]}$. Hence   relation $\equiv$ is transitive.
\end{proof}
\medskip

From the above lemma,  we can consider the quotient set 

$$
({\mathcal{G}} / \equiv)=\{ [A_{[[e_{i}]]}]: A_{[[e_{i}]]}
\in {\mathcal{G}}\},
$$
and for any $[A_{[[e_{i}]]}] \in ({\mathcal{G}} / \equiv)$  denote by $%
A_{[[[e_i]]]}$ the linear subspace

\begin{equation}\label{final1}
A_{[[[e_i]]]}:= \bigoplus\limits_{A_{[[e_{j}]]} \in [%
A_{[[e_{i}]]}]}A_{[[e_{j}]]},
\end{equation}
being so  $A$ the direct sum of linear subspaces

\begin{equation}\label{final3}
{A}=\bigoplus\limits_{[A_{[[e_{i}]]}] \in ({\mathcal{G}} / \equiv)} A_{[[[e_i]]]}.
\end{equation}

Observe  that we also have 

\begin{equation}\label{final2}
A_{[[[e_i]]]}A_{[[[e_j]]]}=0
\end{equation}

when $A_{[[[e_i]]]} \neq A_{[[[e_j]]]}$ as consequence of Equation (\ref{final1}) and Equation (\ref{orto}).


\begin{proposition}\label{Harbin35}
Any of the linear subspaces $A_{[[[e_i]]]}$ given in Equation (\ref{final1}) is an  ideal of $A$.

\end{proposition}

\begin{proof}

Let us fix some $A_{[[[e_i]]]}$ with $[A_{[[e_i]]}] \in  ({\mathcal{G}} / \equiv)$ (see Equation (\ref{final1})). Taking into account Equation (\ref{final3}) and 
Equation (\ref{final2}) we just have to show that 

\begin{equation}\label{final10}
A_{[[[e_i]]]}A_{[[[e_i]]]} \subset A_{[[[e_i]]]}.
\end{equation}

To do that, take $x,y \in A_{[[[e_i]]]}$ such that 

\begin{equation}\label{final15}
0 \neq xy=z\in A.
\end{equation}

Taking into account  Equations (\ref{orto}) and (\ref{final1}), on the one hand we have that $$x,y \in A_{[[e_j]]}$$ with 
\begin{equation}\label{final11}
A_{[[e_j]]}  \subset A_{[[[e_i]]]};
\end{equation}
and on the other hand, (considering also Equation (\ref{final3})), we can express  
\begin{equation}\label{final12}
z=z_{1}+z_{2}+\cdots +z_{n}
\end{equation}
with any $0\neq z_{k} \in A_{[[e_k]]} $  and $A_{[[e_k]]} \neq A_{[[e_t]]}$ when $k \neq t $, $k \in \{1,...,n\}.$
 
From here $$0\neq z_k \in \Pi_{A_{[[e_k]]}}({A}_{[[e_{j}]]}{A}_{[[e_{j}]]})$$
and so $A_{[[e_k]]} \equiv A_{[[e_j]]}$ for any $A_{[[e_k]]}$, $k \in \{1,...,n\}$. By Equation (\ref{final11}), we get $z_k \in A_{[[e_k]]} \subset  A_{[[[e_i]]]}$ for any $k \in \{1,...,n\}$. From here Equation  (\ref{final12}) allows us to assert  $z \in A_{[[[e_i]]]}$ and, taking into account Equation (\ref{final15}), we have that Equation (\ref{final10}) holds. Consequently any $A_{[[[e_i]]]}$ is an ideal of $A$.

\end{proof}
\begin{proposition}\label{Harbin40}
Suppose  ${\rm Ann}({ A}) \neq 0$ and the basis ${\mathcal B}$ is of semi-division. Then any of the ideals  $A_{[[[e_i]]]}$ given in Equation (\ref{final1}) is simple.

\end{proposition}

\begin{proof}

Let us fix a non-zero ideal $I$ of $A_{[[[e_i]]]}$ and let us show that necessarily $I=A_{[[[e_i]]]}.$

\medskip 

To do that, take some $0 \neq x \in I$. Taking into account  ${\rm Ann}({ A}) \neq 0$, we can find  $e_j \in {\mathcal B}$ satisfying  either $xe_j \neq 0$ or $e_j x \neq 0$. Since  $x=\sum\limits_{t=1}^{m} \lambda_t e_t$ with any $e_t \in {\mathcal B}$ and  $\lambda_t \in {\mathbb K} \setminus \{0\}$, we have either

\begin{equation}\label{Harbin20}
 \hbox{   $0 \neq (\sum\limits_{t=1}^{m} \lambda_t e_t) e_j=y \in { I}$ or $0 \neq e_j(\sum\limits_{t=1}^{m} \lambda_t e_t)=z \in { I}.$}
\end{equation}

  \bigskip 

In the first case. That is 

$$0 \neq (\sum\limits_{t=1}^{m} \lambda_t e_t) e_j=y \in { I},$$
by denoting  $$C=\{t \in \{{1},...,{m}\}: e_te_j \neq 0\},$$  we have
$0 \neq (\sum\limits_{s \in C} \lambda_s e_s)e_j=y \in { I}$, being $\sum\limits_{s \in C} \lambda_s e_s\in {\mathcal M}_{e_j}$. Taking now into account that ${\mathcal B}$ is, in particular, a  weak-division basis we get 
$e_j \in {\mathcal I}(y) \subset { I}.$
In the second case we can argue similarly to get $$e_j \in { I}.$$ 

By observing again that ${\mathcal B}$ is a  weak-division basis, we can apply Lemma \ref{penu}  to obtain that 
\begin{equation}\label{hel7}
{ A}_{[e_j]} \subset { I}.
\end{equation}

\medskip

Let us now show that ${ A}_{[[e_j]]} \subset { I}:$

Take  any  ${ A}_{[e_k]} \subset { A}_{[[e_j]]}$ such that ${ A}_{[e_k]} \neq { A}_{[e_j]}$, since ${ A}_{[e_k]} \approx { A}_{[e_j]}$, there exists
\begin{equation}\label{harbin2}
\{{ A}_{[e_1]},{ A}_{[e_2]},...,{ A}_{[e_n]}\}
\end{equation}

such that ${ A}_{[e_1]}={ A}_{[e_j]}$, ${ A}_{[e_n]}={ A}_{[e_k]}$ and  $${ A}_{[e_j]} { A}_{[e_2]}+ { A}_{[e_2]} { A}_{[e_j]} \neq 0.$$ By Equation (\ref{hel7}), we have
$$0 \neq { A}_{[e_j]} { A}_{[e_2]}+ { A}_{[e_2]} { A}_{[e_j]} \subset {I}.$$

From here, there exist $e_r \in { A}_{[e_j]}$ and $e_s \in { A}_{[e_2]} $ such that either $0 \neq e_r e_s=x \in { I}$ or $0 \neq e_s e_r=y \in { I}$. By the weak-division character of ${\mathcal B}$ we get
either $e_s \in {\mathcal I}(x) \subset { I}$ or  $e_s \in {\mathcal I}(y) \subset I$. By applying Lemma  \ref{penu}, in any case,   we get
$${ A}_{[e_2]} \subset { I}.$$
By iterating this argument with Equation (\ref{harbin2}) we get ${ A}_{[e_n]} ={ A}_{[e_k]} \subset { I}.$  From here, since $${ A}_{[[e_j]]}= \bigoplus\limits_{{ A}_{[e_{k}]} \subset { A}_{[[e_j]]}}{ A}_{[e_{k}]}$$ we get
 $$ { A}_{[[e_j]]} \subset { I}.$$ 

\medskip
Finally, let us prove that ${ A}_{[[[e_j]]]} \subset { I}:$

\medskip

To do that, take some  $A_{[[e_k]]}$ such that $A_{[[e_k]]} \subset { A}_{[[[e_j]]]}$. Hence

\begin{equation}\label{Harbin10}
{ A}_{[[e_k]]} \equiv { A}_{[[e_j]]}. 
\end{equation}
We have to show that   ${ A}_{[[e_k]]}  \subset I.$

By Equation (\ref{Harbin40})  there exists a sequence   
\begin{equation}\label{metalli}
({ A}_{[[e_j]]},{ A}_{[[e_2]]},...,{ A}_{[[e_k]]}) 
\end{equation}

with any $A_{[[e_{r}]]} \in {\mathcal G}$ 
 satisfying conditions in Definition \ref{ane3}. From here, either $\Pi_{{ A}_{[[e_2]]}}({ A}_{[[e_j]]}{ A}_{[[e_j]]}) \neq 0$ or $\Pi_{{ A}_{[[e_j]]}}({ A}_{[[e_2]]}{ A}_{[[e_2]]}) \neq 0$.

 \medskip

In  the first case, Equation (\ref{harbin1}) gives us  there exist $x,y \in { A}_{[[e_j]]}$ such that
$0\neq xy=x_r+\cdots+x_2+\cdots+x_n \in I$ with $0 \neq x_2 \in { A}_{[[e_2]]}$ and $x_p \in { A}_{[[e_p]]}$, $p\in \{1,...,n\},$ with ${ A}_{[[e_p]]} \neq { A}_{[[e_q]]}$ when $p \neq q$. Write now $x=\sum\limits_{s=1}^{n} \lambda_s e_s$ and $y=\sum\limits_{t=1}^{m} \lambda_t e_t$ as their unique expression respect the linear basis ${\mathcal B} \cap { A}_{[[e_j]]}$. Since $xy\neq 0$, there exists $e_s$ and $e_t$ in the above expressions of $x$ and $y$ such that   $$0 \neq e_se_t=z_1+\cdots+z_2+ \cdots+z_n \in I$$  with $0\neq z_2 \in { A}_{[[e_2]]}$  and $x_p \in { A}_{[[e_p]]}$, $p\in \{1,...,n\},$ with ${ A}_{[[e_p]]} \neq { A}_{[[e_q]]}$ when $p \neq q$. From here we can write $$z_2=\sum\limits_{l=1}^{q} \lambda_l e_l$$ with any $e_l \in {\mathcal B} \cap { A}_{[[e_2]]} $, $e_{l_1} \neq e_{l_2}$ when $l_1 \neq l_2$ and $0 \neq \lambda_l \in {\mathbb K}.$ Hence 
$$0 \neq e_se_t=z_1+\cdots+(\sum\limits_{l=1}^{q} \lambda_l e_l)+ \cdots+z_n \in I$$ 
with any $0 \neq \lambda_l \in {\mathbb K}$.  So  $p_{e_l}(e_se_t)=\lambda_l \neq 0$ (see Equation (\ref{Harbin15})) and we get that

\begin{equation}\label{Harbin16}
  e_s,e_t \in {\mathcal P}(e_l) 
\end{equation}
for any $l \in \{1,..., q.\}$

Also, taking into account  ${\rm Ann}({ A}) \neq 0$ and Equation (\ref{orto}) there exists $e_i^{\prime} \in {\mathcal B} \cap  { A}_{[[e_2]]}$ such that $z_2 e_i^{\prime} \neq 0$ or 
$ e_i^{\prime} z_2\neq 0$ and so either $e_{l_0}e_i^{\prime} \neq 0$ or 
$ e_i^{\prime} e_{l_0}\neq 0$ for some $l_0 \in \{1,..., q\}$. From here 

\begin{equation}\label{Harbin17}
e_i^{\prime} \in {\mathcal S}_{e_{l_0}}.  
\end{equation}
Also we have
\begin{equation}\label{Harbin18}
0\neq (e_se_t)e_i^{\prime}=\sum\limits_{l=1}^{q} \lambda_l e_le_i^{\prime}=c_1 \in I
\end{equation}
or 
$$0\neq e_i^{\prime}(e_se_t)=\sum\limits_{l=1}^{q} \lambda_l e_i^{\prime}e_l=c_2 \in I
$$
with $\lambda_{l_0}\neq 0.$

Taking now into account Equations (\ref{Harbin16}), (\ref{Harbin17}) and (\ref{Harbin18}), the semi-divison character of ${\mathcal B}$ gives us $$ 0\neq e_i^{\prime}\in I(c_t) \subset I,$$ $t \in \{1,2\}$.

Since $0\neq e_i^{\prime} \in { A}_{[[e_2]]} \cap I$ we can argue as the beginning of the proof to get $$ { A}_{[[e_2]]} \subset  I.$$

By iterating this process with the sequence given by Equation (\ref{Harbin40}) we get that 

$$ { A}_{[[e_k]]} \subset  I$$
and consequently 

$${ A}_{[[[e_j]]]} \subset { I}$$
as desired. 

\medskip

From here $$I={ A}_{[[[e_j]]]}$$ and ${ A}_{[[[e_j]]]}$ is simple.

\bigskip

Now, let us consider the second possibility in Equation (\ref{Harbin20}). That is,

$$0 \neq e_j(\sum\limits_{t=1}^{m} \lambda_t e_t)=z \in { I}.$$

By arguing as in the above first possibility, we have that there exist 
$e_s,e_t \in {\mathcal B} \cap { A}_{[[e_2]]}  $ and 
$e_{l_0}, e_i^{\prime} \in {\mathcal B} \cap { A}_{[[e_k]]}  $ such that 
$e_s,e_t \in {\mathcal P}(e_{l_0}) $,
 $e_i^{\prime} \in {\mathcal S}_{e_{l_0}}$
and  either

\begin{equation}\label{Harbin30}
    \hbox{$0 \neq (e_se_t) e_i^{\prime}=c_1 \in I$ or $0 \neq  e_i^{\prime}(e_se_t)=c_2 \in I$}
\end{equation}

 By applying the semi-division character of ${\mathcal B} $   to Equation (\ref{Harbin30}) we get 
$$e_se_t =c_3\in I(c_i) \subset I$$
for some $i \in \{1,2\}.$

Now, the (in particular) weak-division character of ${\mathcal B} $ gives us $$e_s,e_t \in I(c_3) \subset I.$$

Since $e_t,e_s \in { A}_{[[e_2]]}$ we conclude as in the first possibility that 
$${ A}_{[[e_2]]} \subset { A}_{[[[e_j]]]}.$$
From here  the proof finishes  as in the first possibility.
\end{proof}

\begin{theorem}\label{Harbin60}
Let $A$ be an arbitrary algebra. Then $A$ is semisimple if and only if  ${\rm Ann}({ A}) \neq 0$ and $A$ admits a   semi-division linear basis. 

\end{theorem}

\begin{proof}

Suppose $A$ is semisimple. That is 
\begin{equation}\label{Harbin45}
A=\bigoplus\limits_{j\in J}I_j
\end{equation}
with any $I_j$ a simple ideal of $A$.

By \cite[Lemma 1]{semi} we have that ${\mathrm{\ Ann }}({A})=0 $, and that if  for any $j\in J$ we fix some basis ${\mathcal{B}}_{j}$ of $%
{I}_{j}$ and consider the basis of $A$ given by
${\mathcal{B}}:=\dot{\bigcup\limits_{j\in IJ}}{\mathcal{B%
}}_{i}$, we have that ${\mathcal{B}}$ is a weak-division basis of ${A}$. 

Hence, in order to verify ${\mathcal{B}}$ is
a semi-division basis of $A$ we just have to check that the second condition in Definition \ref{Harbin50} holds. To do that,  let us fix some $e\in {\mathcal{B}}$, being $e\in {%
\mathcal{B}}_{j}$ for just one $j\in J.$ By the decomposition (\ref{Harbin45}), we
have ${I}_{j}{I}_{k}=0$ when $j\neq k$ and so the sets  ${%
\mathcal{S}}_{e}
{\mathcal{P}}_{e}\subset {\mathcal{B}}_{j}$ and so in case $%
0\neq (e_re_s) e_t=c$ or $0\neq  e_t (e_re_s)=c$ for some $e_r,e_s\in {\mathcal{P}}_{e}$ and $e_s \in {\mathcal{S}}_{e}$ then $c\in
{I}_{j}$ and so $e_re_s, e_t\in {I}_{j}={{I}}(c)$ by the
simplicity of ${I}_{j}$. From here ${\mathcal{B}}$ is a semi-division basis of $A$.

\bigskip

Suppose now ${\rm Ann}({ A}) \neq 0$ and $A$ admits a   semi-division linear basis. By Equation (\ref{final1}) and Proposition \ref{Harbin35} we have that  $A$ is the direct sum of the ideals

$${A}=\bigoplus\limits_{[A_{[[e_{i}]]}] \in ({\mathcal{G}} / \equiv)} A_{[[[e_i]]]}.$$

Now Proposition \ref{Harbin40} allows us to assert that any  $A_{[[[e_i]]]}$ is simple. From here $A$ is semisimple.

\end{proof}


\begin{corollary}
Let ${A}$ be an arbitrary algebra. Then ${A}$ is
simple if and only if ${\rm Ann}({ A})=0$  and ${A}$ admits an i-division
basis.
\end{corollary}

\begin{proof}

Suppose $A$ is simple. Since  ${\rm Ann}({ A})$ is an ideal of $A$ then ${\rm Ann}({ A})=0$. Observe now that any basis ${\mathcal B}$ of $A$ is of $i$-division as consequence of the only non-zero ideal of $A$ is $A$.

\medskip

Conversely, suppose now that $A$ admits an $i$-division basis ${\mathcal B}=\{e_i\}_{i \in I}$ and ${\rm Ann}({ A})=0$. First we note that any $i$-division basis is always a semi-division basis. Hence Theorem  
\ref{Harbin60} gives us 

$$A=\bigoplus\limits_{j\in J}I_j$$
with any $I_j$ a simple ideal of $A$ and with ${\mathcal B} \cap I_j$ a basis of $I_j$ for any $j \in J$. In case there exist $j,k \in J$ with $j \neq k$, we can take $0\neq e_j , e_j ^{\prime}\in {\mathcal B} \cap I_j$ such that $e_j e_j^{\prime}=c\neq 0$ (because ${\rm Ann}({ I_j})=0$); and $0\neq e_k \in {\mathcal B} \cap I_k$. Then we have $0\neq e_j(e_j ^{\prime} + e_k)=c \in I_j$. By the $i$-division character of ${\mathcal B}$ we get $e_j ^{\prime} + e_k \in I(c) \subset I_j.$ A contradiction. From here the cardinal of $J$ is one and so $A$ is simple.
\end{proof}



\begin{thebibliography}{99}

\bibitem{Yo1} Barreiro, E.,  Kaygorodov, I., and  Sanchez, J.M.:  $k$-modules over linear spaces $by  n$-linear maps admitting a multiplicative basis
Algebr. Represent. Theory 22(3) (2019),  615–626.




\bibitem{Yo2} Calder\'{o}n, A.J.:  On split Lie algebras with symmetric root systems
Proc. Indian Acad. Sci. Math. Sci. 118(3) (2008),  351–356.

\bibitem{Yo3}  Calder\'{o}n, A.J.: On the structure of graded commutative algebras
Linear Algebra Appl. 447 (2014), 110–118.

\bibitem{semi} Calder\'{o}n, A.J., Hegazi, A.S. and Abdelwahab, H.:
 A characterization of the semisimplity of Lie-type algebras through the existence of certain linear bases. 
Linear Multilinear Algebra 65(9) (2017), 1781–1792.

\bibitem{Yo4} Calder\'{o}n, A.J. and Sanchez, J.M.:   On split Leibniz superalgebras
Linear Algebra Appl. 438(12) (2013),  4709–4725.




\bibitem{Yo5} Cao, Y. and  Chen, L.Y.:  On the structure of graded Leibniz triple systems
Linear Algebra Appl. 496 (2016), 496–509.



\bibitem{Yo6} Khalili, V.: On the structure of graded 3-Lie-Rinehart algebras.
Filomat 38(2) (2024),  369–392.


\bibitem{Yo7}  Zhang, J., Zhang, Ch. and  Cao, Y.: On the structure of split involutive regular Hom-Lie algebras
Oper. Matrices 11(3) (2017),  783–792.










\end{thebibliography}
\end{document}